\documentclass[a4paper,12pt]{article}
\usepackage{amsthm}
\usepackage{extpfeil}

\newlength{\abstractwidth}
\renewcommand{\thanks}[1]{\footnote{#1}} % Use this for footnotes

\newcommand{\be}{\begin{equation}}
\newcommand{\bea}{\begin{eqnarray}}
\newcommand{\eea}{\end{eqnarray}} \newcommand{\ee}{\end{equation}}
\newcommand{\N}{{\cal N}} 
 \def\ba{\begin{eqnarray}}
\def\ea{\end{eqnarray}}

\def\N{{\cal N}}

\def\ra{\rightarrow}

\def\o{\omega}

\def\o{\omega}

\def\al{\alpha}
\def\b{\beta}

\def\o{\omega}

\def\t{\theta}

\def\ti{\tilde}

\def\N{\bf N}

\def\R{{\bf R}}

\def\ra{\rightarrow}

\def\F{{\cal F}}

\def\[{{\bf [}}
\def\]{{\bf ]}}

\def\Ci{{\bf C}^{\infty}}
\def\pl{\partial}

\begin{document}
\newtheorem{theorem}{Theorem}
\newtheorem{proposition}{Proposition}
\newtheorem{lemma}{Lemma}
\newtheorem{corollary}{Corollary}
\newtheorem{definition}{Definition}
\newtheorem{conjecture}{Conjecture}
\newtheorem{example}{Example}
\newtheorem{claim}{Claim}

\begin{centering}
 
\textup{\LARGE\bf{Stable Higgs bundles and Hermitian-Einstein metrics on non-K\"ahler manifolds}}

\vspace{5 mm}

\textnormal{Adam Jacob}

\begin{abstract}
{\small

Let $X$ be a compact Gauduchon manifold, and let $E$ and $V_0$ be holomorphic vector bundles over $X$. Suppose that $E$ is stable when considering all subsheaves preserved by a Higgs field $\t\in H^0({\rm End}(E)\otimes V_0)$. Then a modified version of the Donaldson heat flow converges along a subsequence of times to a solution of a generalized Hermitian-Einstein equation, given by ${i\mkern1mu}\Lambda F+[\t,\t^\dagger]=\lambda I$.}

\end{abstract}

\end{centering}

{\vskip -5mm}
\begin{normalsize}

\medskip
\section{Introduction}

Given a holomorphic vector bundle $E$ over a complex manifold $X$, a natural question is whether it admits a Hermitian-Einstein metric. Existence of such a metric was first proven by Narasimhan and Seshadri in the case of curves $\cite{NS}$, then for algebraic surfaces by Donaldson $\cite{Don1}$, and for higher dimensional compact  K\"ahler manifolds by Uhlenbeck and Yau $\cite{UY}$. Buchdahl extended Donaldson's result to arbitrary compact complex surfaces in $\cite{Buch2}$, and Li and Yau generalized Uhlenbeck and Yau's theorem to any compact complex Hermitian manifold in $\cite{LY}$. In all cases, existence was found to be equivalent to slope stability in the sense of Mumford-Takemoto.

Many generalizations of this result exist, including to the case of Higgs bundles by Simpson in $\cite{Simp}$. We briefly review his result here. Let $X$ be a compact K\"ahler manifold. A Higgs bundle is a vector bundle $E$, together with a holomorphic endomorphism valued one form:
\be
\t:E\longrightarrow \Lambda ^{1,0}(E)\nonumber
\ee
called a Higgs field. We assume the Higgs field satisfies the integrability condition $\t\wedge\t=0$. If $\t^\dagger$ is the adjoint of $\t$ with respect to $H$, and $\nabla$ is the usual unitary-Chern connection on $E$, we can define a new connection $D:=\nabla+\t+\t^\dagger$, and try to solve the Hermitian-Einstein problem:
\be
\label{HE equation1}
{i\mkern1mu}\Lambda F_\t=\lambda I,
\ee
where here $F_\t$ is the curvature of $D$. This leads to solutions of Hermitian-Einstein equation without the restriction that the connection be unitary. Given this setup, using much of the machinery from both the paper of Donaldson \cite{Don1} and Uhlenbeck and Yau $\cite{UY}$,  Simpson was able to construct a solution to $\eqref{HE equation}$ in the case that $E$ is stable. Here stability is defined as before, with the restriction that each subsheaf $\F$ be preserved by the Higgs field.

One of the key applications of Simpson's work is to use a solution of \eqref{HE equation1} to construct a flat bundle. If  $c_1(E)=0$, then equation \eqref{HE equation} reads
\be
\label{F0}
{i\mkern1mu}\Lambda F_\t=0.
\ee
Furthermore, if $c_2(E)=0$, it follows that
\be
0=\int_M{\rm Tr}(F_\t\wedge F_\t)\wedge\o^{n-2}=||F_\t||^2_{L^2}-||{i\mkern1mu}\Lambda F_\t||^2_{L^2},\nonumber
\ee
from which we can conclude that $D$ is a flat connection. Now, in the K\"ahler case this gives one half of the correspondence between stable Higgs bundles and stable representations of the fundamental group. To see the other half, define a flat connection $D$ on $E$ to be stable if $E$ admits no non-trivial $D$-invariant subbundles. Given a metric on $E$ we can decompose the connection $D$ as $D=\nabla+\theta+\theta^\dagger$, where $\nabla$ preserves the metric. Then if $X$ is K\"ahler, using the existence of harmonic metrics \cite{Cor1, DO, Don3}, and a Bochner type formula of Siu $\cite{Siu2}$ and Sampson $\cite{Samp}$, it follows that if $D$ is stable than there exists a metric so that $(\nabla^{0,1})^2=\nabla^{0,1}\t=\t\wedge\t=0$, thus constructing a stable Higgs bundle. This correspondence between Higgs bundles and flat connections has yielded some fascinating geometric and topological results, including restrictions on the fundamental group of compact K\"ahler manifolds. For more details we direct the reader to \cite{Cor2}. 

It is natural to ask if the above results can be generalized to the non-K\"ahler case. In $\cite{Bis}$, by constructing an explicit example, it was shown by Biswas that the direct correspondence between stable Higgs bundles and representations of the fundamental group does not extend to this case. Note that when $X$ is non K\"ahler, the degree of a bundle is not a topological invariant. Thus a solution to \eqref{F0} will only yield a flat bundle if 
\be
\int_M{\rm Tr}(F_\t)\wedge\o^{n-1}=\int_M{\rm Tr}(F_\t\wedge F_\t)\wedge\o^{n-2}=0,\nonumber
\ee
which is a much more restrictive condition than that of vanishing Chern classes. Furthermore, given a stable flat connection, one can only construct a stable Higgs bundle if certain metric invariants called ``pseudo Chern classes" vanish \cite{Lub} (they always vanish if $X$ is K\"ahler, see \cite{Simp2} for details).

Despite the above difficulties, much work has been done to generalize equation \eqref{HE equation1} and study the corresponding moduli of solutions in the non-K\"ahler case. An extremely general correspondence between stable holomorphic paris and solutions of a Hermitian-Einstein type equation was worked out by Teleman and L\"ubke in \cite{TL2}, building off the work of Banfield \cite{Ban}, Mundet i Riera \cite{M}, and Bradlow, Garcia-Prada, and Mundet i Riera \cite{BGM} (among others) in the K\"ahler case. The holomorphic pairs considered consist of a holomorphic vector bundle and a group action. We direct the reader to \cite{TL2} and the references therein for details, and only address the case of Higgs pairs here. This is a special case of the more general setup proven in \cite{TL2, BGM}, yet is still a generalization of equation \eqref{HE equation1}. 

Let $V_0$ be a fixed holomorphic vector bundle with metric $\eta$. Consider the following $V_0$-twisted endomorphism:
\be
\t\in H^0({\rm End}(E)\otimes V_0).\nonumber
\ee
Given any metric $H$, we can take $\t^\dagger$ (the adjoint of the endomorphism part of $\t$), and define an $H$-Hermitian endomorphism of the bundle $E$ by $[\t,\t^\dagger]_\eta,$
given by the standard commutator contracted by the metric $\eta$ on $V_0$. One now looks for a solution to
\be
\label{HE equation}
{i\mkern1mu}\Lambda F+[\t,\t^\dagger]_\eta=\lambda I,
\ee
which again exists if and only if $E$ is stable. In \cite{TL2} Teleman and L\"ubke utilized the elliptic method of continuity (same as \cite{LY}) to solve their equation. In this paper we find a solution of \eqref{HE equation} using a parabolic heat flow method. Specifically, we look at the following non-linear flow on the space of metrics:
\be
\label{DHF}
H^{-1}\dot H=-({i\mkern1mu}\Lambda F+[\t,\t^\dagger]_\eta-\lambda I),
\ee
and prove convergence is dependent upon stability. Our main result is as follows:
\begin{theorem}
\label{maintheorem}
Let $X$ be a compact, complex Hermitian manifold equipped with a Gauduchon metric, and let $V_0$ and $E$ be holomorphic vector bundles over $X$. Assume there exists a Higgs field $\t\in H^0({\rm End}(E)\otimes V_0)$, and that the pair $(E,\t)$ is indecomposable. Then a family of metrics $H(t)$ evolving along \eqref{DHF} converges along a subsequence of times to a solution of \eqref{HE equation} if and only if $(E,\t)$ is stable.
\end{theorem}

 The parabolic approach we follow was used by Donaldson $\cite{Don1}$ and Simpson $\cite{Simp}$ in the K\"ahler case. Since we are focusing on the non-K\"ahler case, some extra care needs to be taken. Aside from having to be careful with additional terms during integrating by parts, the main difficulty we encounter is that Simpson's proof of the $C^0$ estimate for $H$ does not carry over to our case, because the form of the Donaldson functional he uses can not be defined if $X$ is only Gauduchon. Instead, we adapt the elliptic $C^0$ estimate of Uhlenbeck and Yau to our parabolic setting. This step requires careful control of the subsequences taken along to the flow in order to construct a destabilizing subsheaf. Just as in \cite{Simp}, we also need a fundamental theorem of Uhlenbeck and Yau, which states that weakly holomorphic subbundles are in fact holomorphic subsheaves of $E$. Armed with the $C^0$ estimate, we then use parabolic and elliptic methods to gain higher order estimates for $H$, allowing us to prove convergence along a subsequence.

Our result provides the first heat flow proof of the Hermitian-Einstein problem in the non-K\"ahler case. We consider this approach a worthwhile investigation in that we have developed techniques for using geometic flows in this more general setting. Heat flow methods have gained in prominence following Perelman's solution of the Poincar\'e conjecture \cite{P1,P2,P3} using Hamilton's Ricci flow \cite{Ham}, and the study of flows related to the Ricci flow, the mean curvature flow, and the Yang Mills flow (among many others) remains an active branch of current research in differential geometry.

We divide up the paper as follows. Section 2 contains general background material that will be used throughout the subsequent arguments. In Section 3 we introduce the Donaldson heat flow and describe the evolution of certain key quantities.  Section 4 contains the proof of long time existence of the flow as well as the proof of our main result under the assumption that $H$ is bounded in $C^0$. Finally, in Section 5 we show how to achieve the $C^0$ bound for $H$ using the stability of $(E,\t)$. 

\medskip

\begin{centering}
{\bf Acknowledgements}
\end{centering}
\medskip

First and foremost, the author would like to thank his thesis advisor, D.H. Phong, for all his guidance and support during the process of writing this paper. The author also thanks Valentino Tosatti and Gabor Sz\'ekelyhidi for much encouragement and some helpful suggestions. Finally, the author would like to thank the referee for pointing out several errors and making important suggestions to an earlier draft of this paper. The referee also introduced the author to \cite{Ban, BGM, TL2, M}, and suggested generalizing the earlier draft to Higgs pairs, and for this the author is most grateful. This research was funded in part by the National Science Foundation, Grant No. DMS-07-57372, as well as Grant No. DMS-1204155. The results of this paper are part of the author's Ph.D. thesis at Columbia University.

\section{Preliminaries}

\label{HB}

We begin with some basic definitions used throughout the paper. Let $X$ denote a compact Hermitian manifold of complex dimension $n$, and let $g$ be a Hermitian metric on the holomorphic tangent bundle $T^{1,0}X$.  Associated to $g$ one can construct the following fundamental form:
 \be
\omega=\frac{i}{2}\,g_{\bar kj}dz^j\wedge d\bar z^k.\nonumber
\ee
Wedging $\omega$ to the highest power defines the following natural volume form
\be
Vol(X):= \int_X\frac{\omega^n}{n!}.\nonumber
\ee
Let $\Lambda$ denote the adjoint of wedging with $\o$. If $\psi$ is a $(1,1)$ form, then one has the following useful equality
\be
\psi\wedge\frac{\o^{n-1}}{(n-1)!} = (\Lambda \psi)\, \frac{\o^n}{n!}.\nonumber
\ee

We are interested in certain special classes of Hermitian metrics, all of which are defined by properties of $\o$.

\begin{definition}
We say $g$ is  Gauduchon if $\pl\bar\pl(\o^{n-1})=0$, semi-K\"ahler if $d(\o^{n-1})=0$, and K\"ahler if $d(\o)=0$.
\end{definition}
In this paper we focus on metrics which satisfy the Gauduchon condition, which was introduced by Gauduchon in \cite{G}. Although such metrics have much less structure than K\"ahler metrics, they exist in abundance. In fact, any compact Hermitian manifold $X$ admits a Gauduchon metric.

Let $(E,\bar\pl)$ be a holomorphic vector bundle over $X$. Given a metric $H$, every holomorphic bundle admits a Chern connection $d_A$ which preserves the metric and defines the holomorphic structure on $E$. Because $X$ is complex, the Chern connection can be decomposed into $(1,0)$ and $(0,1)$ parts, which we denote by $\pl_A$ and $\bar\pl$. We also denote the Chern connection on the associated bundle End$(E)$ by $\pl_A$ and $\bar\pl$. Furthermore, we use the notation $\nabla=\nabla^{1,0}+\nabla^{0,1}$ to denote the Chern connection on all associated bundles of the form $E\otimes \Omega^{p,q}$ and End$(E)\otimes\Omega^{p,q}$. Thus, when working on $E\otimes \Omega^{0}$ and End$(E)\otimes\Omega^{0}$, one has $\nabla^{1,0}=\pl_A$ and $\nabla^{0,1}=\bar\pl$. However, working on $E\otimes\Omega^{p,q}$, with $p$ or $q$ (or both) nonzero, $\nabla^{1,0}$ and $\nabla^{0,1}$ contain connection terms coming from the bundle $\Omega^{p,q}$ in addition to $E$, while $\pl_A$ and $\bar\pl$ only contain connection terms for $E$. Because $g$ is Gauduchon and not K\"ahler, the Chern connection on $\Omega^{p,q}$ does not coincide with the Levi-Civita of the Riemannian metric, and one must deal with torsion terms when working with $\nabla$ on $\Omega^{p,q}$.

Let $F$ denote the curvature of of the Chern connection on $E$. Since the holomorphic structure $\bar\pl$ on $E$ is fixed, we can view $F$ as depending only on our choice of metric $H$.  

\begin{definition}\label{degree} The degree of the holomorphic bundle $(E,\bar\pl)$ is defined as follows
\be
deg(E)=\frac{{i\mkern1mu}}{2\pi}\int_X{\rm Tr}(F)\wedge\frac{\o^{n-1}}{(n-1)!}.\nonumber
\ee
\end{definition}
Because $g$ is Gauduchon, the above quantity does not depend on a choice of metric for $E$. Given a different metric $\hat H$ on $E$, there is a smooth function $\psi$ on $X$ which satisfies Tr$(F_\t-\hat F_\t)\wedge\o^{n-1}=\pl\bar\pl\psi\wedge\o^{n-1}$, and this integrates to zero in the Gauduchon case. Now, although degree is independent of the metric on the bundle upstairs, it does depend on $g$, and is only a topological invariant if $g$ is K\"ahler or semi-K\"ahler (see \cite{TL} for details). 

A metric is called Hermitian-Einstein if it solves the following equation,
\be
{i\mkern1mu}\Lambda F=\lambda I\nonumber
\ee
where $\lambda$ is a real number. In fact, because Definition \ref{degree} is independent of metric, the constant $\lambda$ is completely specified
\be
\label{lambda}
\lambda=\frac{2\pi deg(E)}{rk(E)Vol(X)}.
\ee
For notational simplicity we denote the $H$-Hermitian endomorphism ${i\mkern1mu}\Lambda F$ by $K$.

As stated in the introduction, in this paper we consider a generalization of the Hermitian Einstein equation. Let $V_0$ be a fixed holomorphic vector bundle with metric $\eta$. Consider the following $V_0$-twisted endomorphism
\be
\t\in H^0({\rm End}(E)\otimes V_0).\nonumber
\ee
In the classical theory of Higgs bundles one takes $V_0=\Omega^1_X$ (see \cite{Simp, Simp2}). Even though we allow for $V_0$ to be arbitrary, we still refer to $\t$ as a Higgs field. For a given metric $H$ we can consider the following section
\be
\t^\dagger\in\Gamma({\rm End}(E)\otimes \bar V_0),\nonumber
\ee 
defined by taking the adjoint of the endomorphism part of $\t$ with respect to the metric $H$. In other words for sections $s$ and $t$ of $E$, $\t^\dagger$ is defined so that the following sections of $V_0$ are equal
\be
\langle \t s,t\rangle_H=\langle s,\t^\dagger t\rangle_H.\nonumber 
\ee 
We now define an $H$-Hermitian endomorphism of the bundle $E$ by $[\t,\t^\dagger]_\eta,$ given by the standard commutator of the endomorphism parts of $\t$ and $\t^\dagger$ contracted by the metric $\eta$ on $V_0$. Note that this quantity varies as $H$ varies. In the K\"ahler case, $[\t,\t^\dagger]_\eta,$ is derived via a moment map construction. However, in our case $X$ is only Gauduchon and therefore not symplectic. As a result we choose not to describe this formalism here and instead direct the reader to \cite{Ban, M} for details.

\begin{definition}
\label{indecom}
We say the pair $(E,\t)$ is indecomposable if $E$ can not be split holomorphically into the direct sum of two subbundles, each of which is preserved by $\t$.
\end{definition}
From this point on we always assume $(E,\t)$ is indecomposable. The generalized Hermitian-Einstein equation we consider is expressed as follows
\be
K+[\t,\t^\dagger]_\eta=\lambda I.\nonumber
\ee
For notational simplicity we set $K_\t=K+[\t,\t^\dagger]_\eta$. Our main theorem is to solve the above equation using parabolic methods. Just as in the elliptic case existence is intimately tied to a notion of geometric stability, which we go over now.

Consider a proper coherent subsheaf $\F\subset E$ with torsion free quotient. 
\begin{definition}
We say $\F$ is a {\it sub-Higgs sheaf} of $E$ if $\t|_\F$ defines an element of $H^0({\rm End}(\F),V_0)$. 
\end{definition}
Since $(\F,\bar\pl)$ is a holomorphic vector bundle away from a singular set $Z(\F)$ of codimension 2, on $X\backslash Z(\F)$ we can consider the orthogonal projection $\pi:E\longrightarrow \F$ defined by $H$. Let $\phi$ be a section of $\F$. The connection $\nabla$ induces a connection on $\F$, which is given by $\nabla_{\F}(\phi)=\pi\circ \nabla(\phi)$. The second fundamental form is a map from $\F$ to its $H$-orthogonal complement $\F^\perp$ defined by
\be
(\nabla-\nabla_\F)\phi=(I-\pi)\nabla\phi.\nonumber
\ee
Because $\bar\pl$ preserves $\F$, we know $(I-\pi)\circ \bar\pl=0$, so the second fundamental form can in fact be expressed as $(I-\pi)\pl_A\phi$. We now compute
\be
\pl_A(\pi)\phi=\pl_A(\pi\phi)-\pi \pl_A(\phi)=(I-\pi)\pl_A\phi.\nonumber
\ee
Thus $\pl_A\pi:\F\rightarrow \F^\perp$ defines the second fundamental form associated to $\nabla$. Now, since the metric $H$ on $E$ defines a metric $H_\F$ on $\F$ over $X\backslash Z(\F)$ by inclusion, one can compute how curvature $F$ of $H$ restricts to $\F$ (see $\cite{GH}$)
\be
F_\F=\pi F\pi-(\pl_A\pi)^\dagger\wedge \pl_A\pi.\nonumber
\ee
Although this formula only holds on $X\backslash Z(\F)$, we know induced curvature is at least in $L^1$ (see \cite{J,Kob}, for instance), and since $Z(\F)$ has zero measure the degree of $\F$ can once again be defined by integrating over $X$
\be
deg(\F)=\frac{i\mkern1mu}{2\pi}\int_X{\rm Tr}(\pi F\pi)\wedge\frac{\o^{n-1}}{(n-1)!}-\frac1{2\pi}||\pl_A\pi||^2_{L^2(H)}.\nonumber
\ee
This is the well know Chern-Weil formula. 
\begin{definition}
We say $(E,\t)$ is stable if, given any proper sub-Higgs sheaf $\F\subset E$ with torsion free quotient, we have 
\be
\mu(\F)=\frac{deg({\cal F})}{rk({\cal F})}<\frac{deg(E)}{rk(E)}=\mu(E).\nonumber
\ee
\end{definition}

\section{The Donaldson heat flow}

\label{sdhf}

In this section we introduce the parabolic equation used to solve \eqref{HE equation}. Because of its similarities with the K\"ahler case, we still refer to the flow as the Donaldson heat flow. Given an initial metric $H_0$, we define the flow of endomorphsims $h=h(t)$ by
\be
h^{-1}\dot h=-(K_\t-\lambda I),\nonumber
\ee
where $h(0)=I$ and $K_\t=K_\t(t)$ is determined by metric $H(t)=H_0h(t)$. The main goal of this paper is to show the flow converges to a solution of $\eqref{HE equation}$ along a subsequence of times. First we compute the evolution of a few key terms. 

We start with the following standard formula, which can be found in \cite{TL, Siu}, and states that if a connection is evolving along a path of metrics, then the time derivative is given by
\be
\label{connection}
\dot \pl_A=\pl_A(h^{-1}\dot h).
\ee
The above formula can be used to compute the time derivative of the curvature endomorphism $K$
\be
\dot K={i\mkern1mu}\Lambda\bar\pl\pl_A(h^{-1}\dot h).\nonumber
\ee
To compute the time derivative of $K_\t$, we need to understand how the endomorphism $[\t,\t^\dagger]_\eta$ evolves. First, we note that although $\t$ is defined to act on sections of $E$, its action can be extended to sections of End$(E)$ by the formula $\t(h)=\t h-h\t$. Similarly we can extend $\t^\dagger$. Note that $\t(I)=\t^\dagger(I)=0$. Now, using the definition of adjoint one can compute
\be
\frac d{d t}(\t^\dagger)=\t^\dagger h^{-1}\dot h-h^{-1}\dot h\t^\dagger=\t^\dagger(h^{-1}\dot h).\nonumber
\ee
As a result we have
 \be
 \label{evolvF}
\frac d{d t}[\t,\t^\dagger]_\eta=[\t,\t^\dagger(h^{-1}\dot h)]_\eta.\nonumber
\ee
Using our flow equation \eqref{DHF}, it follows that
\be
\label{DHFK0}
\dot K_\t={i\mkern1mu}\Lambda \bar\pl\pl_A(h^{-1}\dot h)+[\t,\t^\dagger(h^{-1}\dot h)]_\eta=-{i\mkern1mu}\Lambda \bar\pl\pl_A(K_\t)-[\t,\t^\dagger(K_\t)]_\eta.
\ee

We now define some fully elliptic operators used in the arguments to follow. Consider both
\be
P_A'=i\Lambda \nabla^{1,0}\nabla^{0,1}\qquad{\rm and}\qquad P_A''=-i\Lambda \nabla^{0,1}\nabla^{1,0}.\nonumber
\ee
As in the previous section, the connection $\nabla$ denotes the Chern connection on the associated bundles $E\otimes \Omega^{p,q}$ and End$(E)\otimes\Omega^{p,q}$, and as a result both $P_A'$ and $P_A''$ are defined on these spaces. One could denote the above operators as $P_{A\otimes\Gamma}'$ and $P_{A\otimes\Gamma}''$ in order to specify that $\nabla$ contains connection terms $\Gamma$ for the bundle $\Omega^{p,q}$. However, these connection terms are fixed along the flow, so for notational simplicity we drop $\Gamma$ from our notation. We include the connection $A$ in our notation since it is changing along the flow, and we want to highlight this dependence on time.  Let $P$ to denote the operator $i\Lambda\pl\bar\pl$ on $C^\infty(X)$, where no connection terms are needed. Finally we note that the above operators are defined using the ``analyst convention,'' and are positive definite. 
 
 Returning to \eqref{DHFK0}, the evolution of $K_\t$ can be rewritten as
 \be
 \label{DHFK}
 \dot K_\t=P_A''(K_\t)-[\t,\t^\dagger(K_\t)]_\eta.
 \ee
 
\begin{lemma}
\label{supLF}
Along the Donaldson heat flow, $K_\t$ is uniformly bounded in $C^0$
\be
\sup_{X}|K_\t|_H^2<C.\nonumber
\ee
\end{lemma}
\begin{proof}
First, we remark that the pointwise inner product on endomorphisms of $E$ induced by the metric $H$ is given by
\be
\langle\cdot,\cdot\rangle_H={\rm Tr}(\cdot(\cdot)^\dagger).\nonumber
\ee
This leads the simple observation that 
\be
\langle[\t,\t^\dagger(K_\t)]_\eta, K_\t\rangle_H=|\t(K_\t)|^2_{H\otimes\eta},\nonumber
\ee
which can be seen using the definition of the commutator and properties of trace. Next we note that
\be
\frac d{d t}|K_\t|_H^2=\langle\dot K_\t, K_\t\rangle_H+\langle K_\t, \dot K_\t\rangle_H,\nonumber
\ee
since the contribution of the time derivative of the metric cancels along \eqref{DHF}. Plugging in \eqref{DHFK} yields
\bea
\frac d{d t}|K_\t|_H^2&=&\langle P_A'' K_\t, K_\t\rangle_H+\langle K_\t, P_A''K_\t\rangle_H-\langle[\t,\t^\dagger(K_\t)]_\eta, K_\t\rangle_H-\langle K_\t,[\t,\t^\dagger(K_\t)]_\eta\rangle_H\nonumber\\
&\leq&\langle P_A'' K_\t, K_\t\rangle_H+\langle K_\t, P_A''K_\t\rangle_H.\nonumber
\eea
Note that in this special case $\langle P_A'' K_\t, K_\t\rangle_H=\langle P_A' K_\t, K_\t\rangle_H$, since interchanging the order of derivatives introduces a commutator with $K$, which vanishes under trace. It follows that
\be
\frac d{d t}|K_\t|_H^2\leq\langle P_A' K_\t, K_\t\rangle_H+\langle K_\t, P_A'' K_\t\rangle_H\leq P|K_\t|^2_H.\nonumber
\ee
 The lemma now follows from the maximum principle. Even though $P$ is not equivalent to the standard Laplace-Beltrami operator on functions, the maximum principle applies, as shown by $(7.2.8)$ of \cite{TL}.
\end{proof}
Next we turn to the following normalization lemma.
\begin{lemma}
\label{deth}
We can pick an initial metric on $E$ so that $det(h)=1$ for all time along the flow.
\end{lemma}
The proof of the above lemma is identical to the proof in the K\"ahler case, once one makes the observation that Tr$(K_\t)=$ Tr$(K)$. We direct the reader to Lemma 6 from $\cite{Donovan}$ for details. From this point on, $H_0$ will always be an initial fixed metric on $E$ satisfying Lemma $\ref{deth}$, and $H$ will denote the metric on $E$ evolving along the flow \eqref{DHF}.

We conclude this section with a final computation of the heat operator, which we will need in the analysis to follow. As a first step, we introduce a Bochner type identity for the operator $P_A''$. We define the following Hodge-type Laplacian on  ${\rm End}(E)\otimes\Omega^{p,q}(X)$
\be
\Box=-i([\Lambda, \bar\pl]\pl_A+\pl_A[\Lambda,\bar\pl]).\nonumber
\ee
Because $X$ is not K\"ahler, this operator is not equivalent to the standard Laplace operator $\pl_A\pl_A^\dagger+\pl_A^\dagger\pl_A$, yet is is suitable for our purposes. Let $\b*\gamma$ denote any combination of the tensors $\b$ and $\gamma$, where the exact form is not necessary for future computations. 
\begin{lemma}
\label{Bochner}
For all $\al\in\Gamma({\rm End}(E)\otimes\Omega^{1,0}(X))$ the following Bochner identity holds
\be
\Box\al=P_A''\al+F*\al+R*\al+\nabla^{0,1} T*\al+T*\nabla^{0,1}\al,\nonumber
\ee
where $T$ is the torsion tensor of the Chern connection $\Gamma$ defined by $g$, and $R$ is its curvature.
\end{lemma}
Since many standard references for Bochner type identities only consider K\"ahler manifolds, we include a short proof here for completeness. 
\begin{proof}
First, note that for $\al\in\Gamma({\rm End}(E)\otimes\Omega^{1,0}(X))$, applying $\Box$ gives
\be
\Box\al=-i(\Lambda\bar\pl\pl_A\al+\pl_A\Lambda\bar\pl\al).\nonumber
\ee
Next, in a local coordinate chart we write $\al$ as
\be
\al=\al_j dz^j.\nonumber
\ee
Then in coordinates $-i\Lambda\bar\pl\al$ is explicitly given by $g^{j\bar k}\nabla_{\bar k}\al_j$, and 
\be
-i\pl_A\Lambda\bar\pl\al=(\pl_A)_p(g^{j\bar k}\nabla_{\bar k}\al_j) dz^p=g^{j\bar k}\nabla_p\nabla_{\bar k}\al_j dz^p.\nonumber
\ee
Here we have switched $\pl_A$ to the covariant derivative $\nabla^{1,0}$ since the connection term associated to $\Omega^{1,0}$ comes from the derivative $\pl_A$ landing on the metric $g^{j\bar k}$. Now, in local coordinates the endomorphism valued two from $\pl_A\al$ is given by
\be
\pl_A\al=((\pl_A)_p\al_j-(\pl_A)_j\al_p)dz^p\wedge dz^j.\nonumber
\ee
By introducing the connection terms for $\Omega^{1,0}$, we can rewrite the above expression using covariant derivatives
\be
\pl_A\al=(\nabla_p\al_j-\nabla_j\al_p+\Gamma_{pj}^m\al_m-\Gamma^m_{jp}\al_m)dz^p\wedge dz^j.\nonumber
\ee
Define the torsion tensor $T^m_{pj}$ by $\Gamma_{pj}^m-\Gamma^m_{jp}$. From here we see
\be
-i\Lambda \bar\pl\pl_A\al=g^{j\bar k}(\nabla_{\bar k}\nabla_j\al_p-\nabla_{\bar k}\nabla_p\al_j-\nabla_{\bar k}T^m_{pj}\al_m-T^m_{pj}\nabla_{\bar k}\al_m) dz^p.\nonumber
\ee
Thus putting everything together gives
\bea
\Box\al
&=&\left(g^{j\bar k}\nabla_{\bar k}\nabla_j\al_p-g^{j\bar k}[\nabla_{\bar k},\nabla_p]\al_j-g^{j\bar k}\nabla_{\bar k}T^m_{pj}\al_m-g^{j\bar k}T^m_{pj}\nabla_{\bar k}\al_m\right) dz^p.\nonumber
\eea
The first term above is none other than $P_A''\al$. The commutator term $[\nabla_{\bar k},\nabla_p]$ introduces the curvature terms and the lemma follows. 
\end{proof}

We now use the Bochner identity to prove the following lemma. Because the metric $g$ is fixed throughout the paper, we suppress $g$ from our subscript when denoting the norm of sections of ${\rm End}(E)\otimes \Omega^{p,q}$.

\begin{lemma}
\label{sweetness}
Assume the metric $H$ is bounded uniformly in $C^1$ along the Donaldson heat flow. Then we have the following point wise inequality
\be
\left(\frac d{dt}-P\right)|\pl_A(K_\t)|_H^2\leq C(1+|F|_H)|\pl_A(K_\t)|_H^2-\frac34\left(|\nabla^{0,1}\nabla^{1,0}(K_\t)|_H^2+|\nabla^{1,0}\nabla^{1,0}(K_\t)|^2_H\right).\nonumber
\ee
\end{lemma}
\begin{proof}
 First we compute the time derivative of $|\pl_A(K_\t)|_H^2$
\be
\frac d{dt}|\pl_A(K_\t)|_H^2=2\langle\dot\pl_A K_\t,\pl_A K_\t\rangle_H+2\langle \pl_A\dot K_\t, \pl_A K_\t\rangle_H+\langle [K_\t, \pl_A K_\t],\pl_A K_\t\rangle_H,\nonumber
\ee
where the last term comes from the time derivative hitting the metric $H$. Note the time derivative of the connection $\dot \pl_AK_\t$ is given by the commuator $[\pl_A K_\t, K_\t]$. Since $K_\t$ is bounded uniformly along the flow by Lemma \ref{supLF}, we have
\be
\frac d{dt}|\pl_A(K_\t)|_H^2\leq C|\pl_A(K_\t)|_H^2+2\langle \pl_A(P_A''K_\t), \pl_A K_\t\rangle_H-2\langle \pl_A([\t,\t^\dagger(K_\t)]), \pl_A K_\t\rangle_H.\nonumber
\ee
Note that $\t$ is fixed, and we assumed $H$ is uniformly bounded in $C^1$, giving us control of the connection along the flow and thus control of $\pl_A(\t)$. Furthermore, because our connection is unitary we have $\pl_A(\t^\dagger)=(\bar\pl \t)^\dagger=0$, since we assumed $\t$ to be holomorphic. Thus the final term above is bounded by $C|\pl_A(K_\t)|_H^2$. We turn to the second term on the right hand side above
\be
2\langle \pl_A(P_A'' K_\t), \pl_A K_\t\rangle_H.\nonumber
\ee
We apply the Bochner identity to the endomorphism valued one form $\pl_AK_\t$. Note that
\be
\Box\pl_A K_\t=-i(\Lambda \bar\pl\pl_A\pl_A K_\t+\pl_A\Lambda\bar\pl\pl_A K_\t)=\pl_A(P_A''K_\t)\nonumber
\ee
since $\pl_A^2=0$. Thus by the Bochner identity
\bea
2\langle \pl_A(P_A'' K_\t), \pl_A K_\t\rangle_H=2\langle \Box\pl_A K_\t, \pl_A K_\t\rangle_H&\leq&2\langle P_A''\pl_A K_\t, \pl_A K_\t\rangle_H+C|F|_H|\pl_A K_\t|_H^2\nonumber\\
&&+C|\pl_A K_\t|_H^2+C|\nabla^{0,1}\nabla^{1,0}K_\t|_H|\pl_AK_\t|_H\nonumber,
\eea
where for the last term we applied the Cauchy-Schwarz inequality to 
\be
\langle T*\nabla^{0,1}\pl_A K_\t, \pl_A K_\t\rangle_H.\nonumber
\ee
Now, applying Young's inequality $ab\leq a^2+b^2$ to $a=\frac12|\nabla^{0,1}\nabla^{1,0}K_\t|_H$ and $b=2C|\pl_AK_\t|_H$ gives
\be
C|\nabla^{0,1}\nabla^{1,0}K_\t|_H|\pl_AK_\t|_H\leq \frac1{4}|\nabla^{0,1}\nabla^{1,0}K_\t|_H^2+4C|\pl_A K_\t|_H^2.\nonumber
\ee
Putting everything together so far we see
\be
\frac d{dt}|\pl_A(K_\t)|_H^2\leq  2\langle P_A''\pl_A K_\t, \pl_A K_\t\rangle_H+C(1+|F|_H)|\pl_A(K_\t)|_H^2+\frac1{4}|\nabla^{0,1}\nabla^{1,0}K_\t|_H^2.\nonumber
\ee
Next we compute $P$ on $|\pl_A(K_\t)|_H^2$.
\bea
P\langle\pl_A K_\t,\pl_A K_\t\rangle_H&=&\langle P_A'\pl_A K_\t,\pl_AK_\t\rangle_H+|\nabla^{0,1}\pl_A K_\t|^2_{H}+|\nabla^{1,0}\pl_A K_\t|^2_{H}\nonumber\\
&&+\langle\pl_A K_\t,P_A''\pl_A K_\t\rangle_H\nonumber.
\eea
Note that interchanging $P_A'$ and $P_A''$ introduces terms with $K$, but since $K_\t$ is uniformly controlled and $h$ is in $C^0$ these extra terms can be absorbed into the $C|\pl_AK_\t|^2_H$ term. Thus
\be
\left(\frac d{dt}-P\right)|\pl_A(K_\t)|_H^2\leq C(1+|F|_H)|\pl_A(K_\t)|_H^2-\frac{3}{4}|\nabla^{0,1}\nabla^{1,0}K_\t|_H^2-|\nabla^{1,0}\nabla^{1,0} K_\t|^2_{H},\nonumber
\ee
and the lemma follows.
\end{proof}

\section{Convergence properties of the flow}

As stated in the introduction, the goal of this section is to prove Theorem $\ref{maintheorem}$ under the assumption that Tr$(h)$ is bounded in $C^0$ uniformly in time.  For the remainder of this section, we always assume our initial metric $H_0$ was chosen so that det$(h)=1$ along the flow. As a result, the bound on Tr$(h)$ implies every eigenvalue $\lambda_i$ of $h$ satisfies $0<c\leq \lambda_i\leq C$ uniformly. Many of the results in this section carry over from standard parabolic theory and the results of \cite{Donovan} with minor modifications. We include the important details here for the reader's convenience. 

We begin with the following Proposition, which sums up the key estimates needed to prove long time existence of the flow as well as Theorem $\ref{maintheorem}$.
\begin{proposition}
\label{Higher}
Let $H=H(t)$ be a solution of the Donaldson heat flow in the time interval $[0,T)$, where $T$ can either be a finite time or infinity. If there exists a constant $C_T$ so that ${\rm Tr}( h)\leq C_T$ for all $t\in[0,T)$, then for every $k\in \N$ there exists a constant $A_{k,T}$, depending only on $C_T$, $k$, and fixed initial data, so that $|h|_{C^k}\leq A_{k,T}$. \end{proposition}
Thus once we bound ${\rm Tr}( h)$ in $C^0$, all the higher derivative bounds for $h$ follow. The above proposition is proven in several steps, which are given below. Unless otherwise noted, $T$ can either be taken to be finite or infinity. 
\begin{proposition}
\label{boundS}
If ${\rm Tr}( h)\leq C_T$, then $|\pl_A h h^{-1}|_{H_0}^2\leq C$ for a constant $C$ depending only on $C_T$ and fixed initial data.
\end{proposition}
As shown in \cite{Siu}, $\pl_A h h^{-1}$ measures the difference of the two Chern connections $A-A_0$. The proof of the above proposition consists of several local computations and an application of the maximum principle. Because it does not make use of the global structure of $X$, the proof for $g$ Gauduchon follows from the K\"ahler case (see Section 3.2.1 of \cite{Donovan}), aside from two small details, which we now explain. Observe that the presence of the Higgs field creates an extra term on the right hand side of line (3.2.12) from  \cite{Donovan}, given by $\nabla_j [\t,\t^\dagger]_\eta$. As in the proof of Lemma \ref{sweetness}, $\nabla_j\t^\dagger=0$, so we only have to worry about the contribution of $\nabla_j\t$. However, $\t$ is fixed, so this term is controlled by the connection $A$, and can thus be bounded by $C|\pl_A h h^{-1}|_{H_0}$ and absorbed into an existing term. Furthermore, the application of the second Bianchi identity creates a torsion term, which can be dealt with using the Young's inequality trick from Lemma \ref{sweetness}. With this, the rest of the proof follows as in \cite{Donovan}.

\begin{proposition}
\label{2p}
If ${\rm Tr}( h)\leq C_T$ , then for any $1\leq p<\infty$ we have the following $W^{2,p}$ bound on the for $h$
\be
||h||_{W^{2,p}(H_0)}<C,\nonumber
\ee
where $C$ only depends on $C_T$ and fixed initial data. 
\end{proposition}

\begin{proof}
We begin the proof by recalling the standard formula relating the curvatures of different unitary-Chern connections (see \cite{TL, Siu})
\bea
\label{diffK}
K-K_0&=&{i\mkern1mu}\Lambda\bar\pl(h^{-1} \pl_{A_0}h)\\
&=&h^{-1}{i\mkern1mu}\Lambda\bar\pl\pl_{A_0}h-{i\mkern1mu}\Lambda h^{-1}\bar\pl hh^{-1}\pl_{A_0}h.\nonumber
\eea
Thus we have
\be
\label{bootstrap}
-P_{A_0}'' h=h(K_\t-K_0{}_\t)-h[\t,\t^\dagger]_\eta+h[\t,\t^{\dagger_0}]_\eta+{i\mkern1mu}\Lambda h^{-1}\bar\pl hh^{-1}\pl_{A_0}h.
\ee
An application of Lemma \ref{supLF}, Proposition \ref{boundS}, and the $C^0$ bound for $h$ proves the right hand side above is uniformly bounded in $C^0$. As a result $|P_{A_0}'' h|^2_{H_0}<C$. The proposition now follows from standard $L^p$ theory of elliptic PDE's.
\end{proof}

The preceding proposition shows that the curvature $F$ defined by $H$ is bounded in $L^p$ for any $p$. However, this does not extend to $p=\infty$, which we need for convergence. Thus we must work harder to prove higher regularity.

Define the function $Y(t):\R^+\ra \R^+$ by
\be
Y(t)=||K_\t-\lambda I||^2_{L^2(H)}.\nonumber
\ee
This function will play in important role in long time convergence of the flow. We need the following lemma
\begin{lemma}
\label{limit}
The function $Y(t)$ is non-increasing.
\end{lemma}
\begin{proof}
We begin by computing the time derivative of $Y(t)$
\bea
\dot Y(t)&=&2\int_X{\rm Tr}\left((K_\t-\lambda I)P_A''(K_\t)\right)\frac{\o^n}{n!}-2\int_X{\rm Tr}((K_\t-\lambda I)[\t,\t^\dagger(K_\t)]_\eta)\frac{\o^n}{n!}.\nonumber
\eea
Note that Tr$([\t,\t^\dagger(K_\t)])=0$. This fact, combined with the observation from Lemma \ref{supLF} that 
\be
\langle[\t,\t^\dagger(K_\t)]_\eta, K_\t\rangle_H=|\t(K_\t)|^2_{H\otimes\eta},\nonumber
\ee
gives the following
\bea
\dot Y(t)&=&2\int_X{\rm Tr}\left((K_\t-\lambda I)P_A''(K_\t)\right)\frac{\o^n}{n!}-2||\t(K_\t)||^2_{L^2(H\otimes\eta)}\nonumber\\
&=&-2||\bar\pl(K_\t)||^2_{L^2(H)}-2||\t(K_\t)||^2_{L^2(H\otimes\eta)}+2{i\mkern1mu}\int_X{\rm Tr}\left((K_\t-\lambda I)\pl_A(K_\t)\right)\wedge \frac{\bar\pl(\o^{n-1})}{(n-1)!}.\nonumber
\eea
The second line above follows from integration by parts. We need to show the second term on the right equals zero. To see this we apply the definition of Gauduchon
\bea
0&=&\int_X{\rm Tr}\left((K_\t-\lambda I)^2\right)\,\pl\bar\pl(\o^{n-1})\nonumber\\
&=&-\int_X\pl\,{\rm Tr}\left((K_\t-\lambda I)^2\right)\wedge \bar\pl(\o^{n-1})\nonumber\\
&=&-2\int_X{\rm Tr}\left((K_\t-\lambda I)\pl_A(K_\t)\right)\wedge\bar \pl(\o^{n-1}).\nonumber
\eea
Note that since our connection is unitary and $K_\t$ is self adjoint $||\bar\pl(K_\t)||^2_{L^2(H)}=||\pl_A(K_\t)||^2_{L^2(H)}$. Also, recall that the action of $\t$ on the endomorphism $K_\t$ is given by the commutator, so we have $||\t(K_\t)||^2_{L^2(H\otimes\eta)}=||[\t,K_\t]||^2_{L^2(H\otimes\eta)}$. As a result we see
\be
\label{Ydot}
\dot Y(t)=-2||\pl_A(K_\t)||^2_{L^2(H)}-2||[\t,K_\t]||^2_{L^2(H\otimes\eta)}\leq 0,
\ee
completing the proof of the Lemma
\end{proof}
This leads us to the following important proposition.
\begin{proposition}
\label{limit}
 If ${\rm Tr}( h)\leq C_T$, then $\pl_AK_\t$ is bounded in $L^2$ by a constant only depending on $C_T$ and fixed initial data. Furthermore, if $T=\infty$ and {\rm Tr}$(h)$ is bounded in $C^0$ for all time, then both $||\pl_A(K_\t)||^2_{L^2(H)}$ and $||[\t, K_\t]||^2_{L^2(H\otimes\eta)}$ approach zero as $t$ approaches infinity. 

\end{proposition}
\begin{proof}
Our first step is to prove the following differential inequality 
\be
\label{difft}
\frac d{dt}||\pl_A(K_\t)||^2_{L^2(H)}\leq C||\pl_A(K_\t)||^2_{L^2(H)}+C.
\ee
To begin, we integrate the main inequality from Lemma \ref{sweetness}, noting that the integral of $P|\pl_A(K_\t)|_H^2$ vanishes since $X$ is Gauduchon. This gives
\bea
\frac d{dt}||\pl_A(K_\t)||^2_{L^2(H)}&\leq& C||\pl_A(K_\t)||^2_{L^2(H)}+C\int_X|F|_H|\pl_A(K_\t)|_H^2\frac{\o^n}{n!}\nonumber\\
&&-\frac{3}{4}\left(||\nabla^{0,1}\nabla^{1,0}K_\t||_{L^2(H)}^2-||\nabla^{1,0}\nabla^{1,0} K_\t||^2_{L^2(H)}\right).\nonumber
\eea
Applying H\"older's inequality to the second term on the right hand side yields
\be
\label{interstep}
\int_X|F|_H|\pl_A(K_\t)|_H^2\frac{\o^n}{n!}\leq ||F||_{L^3(H)}||\pl_A(K_\t)||^2_{L^3(H)}.
\ee
Our assumption $|{\rm Tr}( h)|_{C^0}\leq C_T$ implies $h$ is in $W^{2,p}$ by Proposition \ref {2p}, which implies $ ||F||_{L^3(H)}$ is uniformly bounded in time. For notational simplicity let $|\nabla^{0,1}\nabla^{1,0}K_\t|_H+|\nabla^{1,0}\nabla^{1,0} K_\t|_H$ be denoted by $|\nabla^2K_\t|_H$, an expression which controls all second order derivatives of $K_\t$. We now prove an interpolation inequality, similar to that of Hamilton from \cite{Ham}, in order to conclude
\be
\label{interpolation}
||\pl_A(K_\t)||^2_{L^3(H)}\leq  C||K_\t||_{L^6(H)}\left(||\pl_A(K_\t)||_{L^2(H)}+||\nabla^2(K_\t)||_{L^2(H)}\right).
\ee
To see the above inequality, we first integrate by parts
\bea
\int_X|\pl_A K_\t|_H^3\frac{\o^n}{n!}&=&i\int_X|\pl_A K_\t|_H{\rm Tr}\left(\pl_AK_\t(\pl_A K_\t)^\dagger\right)\wedge\frac{\o^{n-1}}{{n-1}!}\nonumber\\
&=&-i\int_X\left(\pl|\pl_A K_\t|_H{\rm Tr}\left(K_\t(\pl_A K_\t)^\dagger\right)+|\pl_A K_\t|_H{\rm Tr}\left(K_\t(\bar\pl\pl_A K_\t)^\dagger\right)\right)\wedge\frac{\o^{n-1}}{{n-1}!}\nonumber\\
&&+i\int_X|\pl_A K_\t|_H{\rm Tr}\left(K_\t(\pl_A K_\t)^\dagger\right)\wedge\pl\left(\frac{\o^{n-1}}{{n-1}!}\right).\nonumber
\eea
The last term on the right introduces a torsion term, which is fixed and controlled by a constant $C$. Furthermore, by Kato's inequality $\pl|\pl_A K_\t|_H\leq |\nabla^{1,0}\nabla^{1,0}K_\t|_H$. Putting these two facts together gives
\be
\int_X|\pl_A K_\t|_H^3\frac{\o^n}{n!}\leq C\int_X|K_\t|_H|\pl_AK_\t|_H|\pl_A K_\t|_H\frac{\o^n}{n!}+ C\int_X|K_\t|_H|\pl_AK_\t|_H|\nabla^2K_\t|_H\frac{\o^n}{n!}.\nonumber
\ee
Applying H\"older's inequality to both integrals on the right we see
\be
\int_X|\pl_A K_\t|_H^3\frac{\o^n}{n!}\leq C||K_\t||_{L^6}||\pl_A K_\t||_{L^3}||\pl_AK_\t||_{L^2}+ C||K_\t||_{L^6}||\pl_AK_\t||_{L^3}||\nabla^2K_\t||_{L^2}.\nonumber
\ee
Dividing both sides by $||\pl_AK_\t||_{L^3}$ proves \eqref{interpolation}. Combining \eqref{interstep} and \eqref{interpolation}, along with Lemma \ref{supLF}, gives
\bea
C\int_X|F|_H|\pl_A(K_\t)|_H^2\frac{\o^n}{n!}&\leq&  C\left(||\pl_A(K_\t)||_{L^2(H)}+||\nabla^2(K_\t)||_{L^2(H)}\right)\nonumber\\
&\leq&||\pl_A(K_\t)||_{L^2(H)}^2+\frac14||\nabla^2(K_\t)||_{L^2(H)}^2+4C^2,\nonumber
\eea
which implies \eqref{difft}.

To achieve the desired $L^2$ bound for $\pl_AK_\t$ for finite time $T$, note that equation \eqref{difft} implies that the function $||\pl_A(K_\t)||^2_{L^2(H)}$ grows at most exponentially in time (see Proposition 8 from \cite{J} for details), giving the following bound:
\be
||\pl_A(K_\t)||^2_{L^2(H)}\leq C e^{kT},\nonumber
\ee
for constants $k$, $C$ only depending on $C_T$ and fixed initial data. The $L^2$ bound for $\pl_AK_\t$ for time $T=\infty$ follows from the second part of the proposition, namely that $||\pl_A(K_\t)||^2_{L^2(H)}$ and $||[\t, K_\t]||^2_{L^2(H\otimes\eta)}$ approach zero as $t$ approaches infinity, which we now demonstrate. 

Define the function $f(t)=||\pl_A(K_\t)||^2_{L^2(H)}+||[\t, K_\t]||^2_{L^2(H\otimes\eta)},$ and assume a solution to the Donaldson heat flow exists for all time. We now integrate $f(t)$ in time from zero to infinity. By \eqref{Ydot} we have
\be
\int_0^\infty f(t)dt=-\frac12 \int_0^\infty \dot Y(t)dt=\frac12Y(0)-\frac12\lim_{b\ra\infty}Y(b).\nonumber
\ee
Since $Y(t)$ is positive the right hand side of the above inequality is bounded. Thus there must exist a subsequence of times $t_k$, such that $t_k<t_{k+1}<t_k+2$, where $f(t_k)$ goes to zero. In fact, if we can can demonstate
\be
\dot f\leq Cf+C,\nonumber
\ee
then it will follow that $f(t)$ goes to zero along any subsequence of times (see \cite{PSSW} for details). Given \eqref{difft}, to prove the above inequality is suffices to show
\be
\label{show this}
\pl_t ||[\t,K_\t]||^2_{L^2(H\otimes\eta)}\leq C||\pl_A(K_\t)||^2_{L^2(H)}+C.\nonumber
\ee
First we compute the time derivative of $|[\t,K_\t]|_{H\otimes\eta}^2.$
\be
\frac d{dt}|[\t,K_\t]|_{H\otimes\eta}^2=\langle[\t,\dot K_\t],[\t,K_\t]\rangle_{H\otimes\eta}+\langle [\t,K_\t], \[\t,\dot K_\t]\rangle_{H\otimes\eta}+\langle [K_\t, [\t,K_\t]],[\t,K_\t]\rangle_{H\otimes\eta}.\nonumber
\ee
The last term on the right, which comes from the time derivative hitting the metric $H$, is bounded by a constant $C$ by Lemma \ref{supLF}. This gives
\be
\frac d{dt}|[\t,K_\t]|_{H\otimes\eta}^2\leq \langle[\t,P_A'' K_\t-[\t,\t^\dagger(K_\t)]],[\t,K_\t]\rangle_{H\otimes\eta}+\langle [\t,K_\t], [\t,P_A'' K_\t-[\t,\t^\dagger(K_\t)]]\rangle_{H\otimes\eta}+C.\nonumber
\ee
Again by Lemma \ref{supLF} and the $C^0$ bound for $h$ the terms involving $[\t,\t^\dagger(K_\t)]$ are controlled. Furthermore $[\t,\t^\dagger]_\eta$ is controlled, so we can change $P_A''$ to $P_A'$ at the cost  of introducing a commutator with $K$, which is now bounded by Lemma \ref{supLF}. Thus we have
\be
\frac d{dt}|[\t,K_\t]|_{H\otimes\eta}^2\leq \langle[\t,P_A'K_\t],[\t,K_\t]\rangle_{H\otimes\eta}+\langle [\t,K_\t], [\t,P_A''K_\t]\rangle_{H\otimes\eta}+C.\nonumber
\ee
We now apply the operator $P$ to $|[\t,K_\t]|_{H\otimes\eta}^2$ 
\be
P|[\t,K_\t]|_{H\otimes\eta}^2\leq\langle P_A'[\t, K_\t],[\t,K_\t]\rangle_{H\otimes\eta}+\langle [\t,K_\t], P_A''[\t, K_\t]\rangle_{H\otimes\eta}.\nonumber
\ee
When computing $P_A'[\t,K_\t]$, we get three types of terms. First, the terms where both derivatives land on $\t$. Since $\t$ is holomorphic, $P_A'\t=g^{j\bar k}\nabla_j\nabla_{\bar k}\t=0$. Thus
\be
P_A''\t=P_A'\t+[K,\t]=[K,\t],\nonumber
\ee
and these terms are controlled. Second, we get mixed terms, where one derivative lands on $\t$ and another lands on $K_\t$. We know $\nabla\t$ is bounded since $\t$ is fixed and the connection term involves at most one derivative of $H$, and $H$ is in $C^1$ by Proposition \ref{boundS}. Thus these terms are controlled by $C|\pl_A K_t|^2_{H\otimes\eta}$. Finally, if both derivatives land on $K_\t$, then we get $[\t,P_A'K_\t]$, which precisely cancels with the time derivate terms. Putting everything together we see
\be
\left(\frac d{dt}-P\right)|[\t,K_\t]|_{H\otimes\eta}^2\leq C|\pl_A K_\t|^2_{H\otimes\eta}+C.\nonumber
\ee
Integrating the above inequality and applying the Gauduchon condition proves \eqref{show this}. Thus $f(t)$ goes to zero strongly in $L^2$, and the proof of the proposition is complete.

\end{proof}

We can now use the $L^2$ bound for $\pl_AK_\t$ to show that in fact $|\pl_A K_\t|^2_{H}$ is bounded in $C^0$. In Lemma \ref{sweetness} we saw
\be
\left(\frac d{dt}-P\right)|\pl_A(K_\t)|_H^2\leq C(1+|F|_H)|\pl_A(K_\t)|_H^2.\nonumber
\ee
Since $C(1+|F|_H)$ is bounded in $L^p$ for any $1<p<\infty$, one can follow the exact parabolic Moser iteration argument from \cite{CJ} to prove $|\pl_A K_\t|^2_{H}$ is bounded in $C^0$.

We have thus shown $K_\t$ is bounded in $C^1$. Furthermore, because $h\in W^{2,p}$ for any $p$, by the Sobolev embedding theorem $h\in C^{1,\al}$ for $\al>0$. Thus, returning to \eqref{bootstrap}, we see that $P_{A_0}'' h$ is bounded in $C^\al$. As a result $h\in C^{2,\al}$, which implies $F\in C^\al$. In fact, once we have $F\in C^0$, higher order bounds for $h$ can be achieved by following standard parabolic theory (see \cite{W2} and the argument given in \cite{Donovan}). For the sake of completeness, we provide a short sketch of the higher order estimates following the same outline as our previous arguments.

So far we proven that $h\in C^{2,\al}$ and $K_\t\in C^1$, using the fact that $h\in C^{1,\al}$ and $K_\t\in C^0$. To obtain higher order estimates, we assume $h\in C^{k,\al}$ and $K_\t\in C^{k-1}$. By equation \eqref{bootstrap} we see that $h\in W^{k+1,p}$ for any $1\leq p<\infty$. Recall that $\nabla=\nabla^{1,0}+\nabla^{0,1}$ denotes the Chern connection on all associated bundles of $E$. As a result $\nabla^k$ denotes taking $k$ covarient derivatives, where each derivative includes the appropriate connection terms for the space it is acting on. One can prove the following inequality 
\be
\label{inequality123}
\left(\frac d{dt}-P\right)|\nabla^k(K_\t)|^2_H\leq  C(1+|\nabla^{k-1} F|_H)|\nabla^k(K_\t)|^2_H-\frac34|\nabla^{k+1}(K_\t)|^2_H
\ee
using the same argument as the one given in Lemma \ref{sweetness}. The main idea is that when computing $\frac d{dt}|\nabla^k(K_\t)|^2_H$, the time derivative can land on either connection terms or $K_\t$. The time derivative of the connection terms produces terms with at most $k$ derives on $K_\t$, which are controlled. When the time derivative hits $K_\t$ it produces a term of the form $\nabla^kP_A''K_\t$. Interchanging the order of derivatives (which lies at the heart of the Bocher identity from Lemma \ref{Bochner}) produces curvature terms and torsion terms, and a term of the form $P_A''\nabla^kK_\t$ which cancels when subtracting $P|\nabla^k(K_\t)|^2_H$. The most derivatives that can land on the curvature $F$ is $k-1$, and the highest order torsion terms can be controlled using the Young's inequality trick from Lemma \ref{sweetness}. All lower order terms are bounded by assumption, and the proof of \eqref{inequality123} follows. Now, the $W^{k+1,p}$ bound for $h$ implies that $C(1+|\nabla^{k-1} F|_H)$ is in $L^p$. Thus, to conclude $K_\t\in C^k$ via parabolic Moser iteration, we only need to show $\nabla^k(K_\t)$ is in $L^2$. Integrating \eqref{inequality123} and applying the higher order analogue of the interpolation inequality \eqref{interpolation} a few times (see \cite{Ham, Donovan}), allows one to prove
\be
\frac d{dt}||\nabla^k(K_\t)||^2_{L^2(H)}\leq -c ||\nabla^k(K_\t)||^2_{L^2(H)}+C\nonumber
\ee
(see the proof of Lemma 10 from \cite{Donovan} for details). Thus $\nabla^k(K_\t)$ is in $L^2$ for all time, and applying parabolic Moser iteration to \eqref{inequality123} gives $|\nabla^k(K_\t)_H|$ is bounded in $C^0$. Thus we have $K_\t\in C^k$, and trivially $K_\t\in C^{k-1,\al}$, so equation \eqref{bootstrap} gives $h\in C^{k+1,\al}$, completing the bootstrap step. We have thus proven Proposition \ref{Higher}.

We can now prove long time existence
\begin{proposition}
Let $H_0$ be an initial metric suitably normalized so that ${\rm det}(h)=1$ along a solution to the Donaldson heat flow \eqref{DHF}. Then a solution to \eqref{DHF} exists for all time $t\in[0,\infty)$.  
\begin{proof}
Because equation \eqref{DHF} is fully parabolic, a solution exists for short time by standard parabolic theory \cite{Lib}. Thus we need to prove long time existence. Suppose that a solution only exists for $t\in[0,T)$ for some finite time $T$. Furthermore suppose $H(t)$ converges in $C^0$ to a metric $H_T$ as $t\ra T$. 
Then ${\rm Tr}(h)\leq C_T$ for some constant $C_T$ independent of $T$, and by Proposition \ref{Higher} we have bounds for all higher order derivatives of $H$. Thus, by taking subsequences, we have smooth convergence of $H(t)$ to $H_T$ and as a result $H_T$ is smooth. Short time existence now allows us to continue the flow to the interval $[0, T+\epsilon).$ 

To see that $H(t)$ converges in $C^0$ to a metric $H_T$ as $t\ra T$, we direct the reader to Proposition 13 from \cite{Don1} as this proof carries over to the Gauduchon case. Additionally, one can see the $C^0$ bound for ${\rm Tr}(h)$ for finite time directly from the flow equation \eqref{DHF}. We have
\be
\frac{d}{dt}{\rm Tr}(h)=-{\rm Tr}\left(h(K_\t-\lambda I)\right)\leq |h|_H|K_\t-\lambda I|_H\leq C{\rm Tr}(h),\nonumber
\ee
using Lemma \ref{supLF} and the fact that all the eigenvalues of $h$ are positive. Then as in the proof of Proposition \ref{limit}, ${\rm Tr}(h)$ grows at most exponentially and is thus bounded for finite time. 
\end{proof}

\end{proposition}
We are now ready to prove Theorem $\ref{maintheorem}$, under the assumption that ${\rm Tr}(h)$ is bounded in $C^0$ for all time.
\begin{proof}
Let $t_i$ be a subsequence of times along the Donaldson heat flow. We assume that there exists a constant $C_\infty$ independent of time so that ${\rm Tr}(h)\leq C_\infty$. Then by Proposition \ref{Higher} we know there exists constants $A_{k,\infty}$ so that $|h|_{C^k}\leq A_{k,\infty}$ for every $k\in \N$.  Thus for each $k$ by the Arzel\`a-Ascoli theorem the metrics $H(t_i)=H_0h(t_i)$ converge in $C^{k-1}$ along a subsequence (still denoted $t_i$) to a limiting metric $H(t_\infty)$. Higher order derivates of this limiting metric are well defined, so in particular $K_\t(t_\infty)$ is well defined and 
\be
\pl_{A_i}K_\t(t_i)\longrightarrow \pl_{A_\infty}{K_\t(t_\infty)}\nonumber
\ee
in $C^{k-4}$. By Proposition \ref{limit} both $||\pl_{A_i}(K_\t(t_i))||_{L^2(H)}^2$ and $||[\t, K_\t(t_i)]||^2_{L^2(H\otimes\eta)}$ go to zero strongly, which implies $\pl_{A_\infty}(K_\t(t_\infty))=[\t, K_\t(t_\infty)]=0$. In fact, because $K_\t(t_i)$ is Hermitian with respect to $H(t_i)$, we also see that $\bar\pl (K_\t(t_\infty))=0$. We now sketch a short proof that stability implies $K_\t$ is a constant multiple of the identity.

Let $u$ be a locally constant Hermitian endomorphism which satisfies $[u,\t]=0$. Assume via contradiction that $u$ is not a constant multiple of the identity. Then there exists an eigenvalue $a$ of $u$ such that $f:=u-aI$ is nonzero. Since $f$ is locally constant both the image $Im(f)$  and the kernel $Ker(f)$ are proper holomorphic subbundles of $E$, and because $[f,\t]=0$ we know that both subbundes are preserved by the Higgs field $\t$. This violates stability, since we can identify $Im(f)$ with $E\backslash Ker(f)$, and as a result it is impossible for both $Ker(f)$ and $Im(f)$ to have slope strictly less than the slope of $E$. 

It follows that $K_\t(t_\infty)$ is a constant multiple of the identity. Because degree is independent of metric this constant multiplier must be $\lambda$. Thus we have constructed a solution to \eqref{HE equation}.

For the ``only if'' part of the proof of Theorem $\ref{maintheorem}$, suppose that the Donaldson heat flow converges along a subsequence of times to a solution of \eqref{HE equation}. Then the stability of the pair $(E,\t)$ is a special case of Theorem 3.3 from \cite{TL2}. 
\end{proof}

\section{The $C^0$ bound from stability}
In this section we prove Tr$(h)$ is uniformly bounded in $C^0$ along the Donaldson heat flow, under the assumptions that $g$ is Gauduchon and $(E,\t)$ is stable. This step is perhaps the most geometrically meaningful, since we have to use the algebraic-geometric condition of stability to prove a uniform bound along a PDE. Simpson proves this bound in the K\"ahler case in Proposition 5.3 from $\cite{Simp}$. Let $M(t):=M(H_0,H(t))$ denote Donaldson's functional (see \cite{Don1,Simp,Kob} for details) along the path of metrics $H(t)$. Simpson proves the following:
\begin{proposition}
Let $E$ be a stable vector bundle over a K\"ahler manifold $X$. If $h(t)=e^{s(t)}$ evolves by the Donaldson heat flow, then for all time
\be
\label{simps}
\sup_X|s|\leq C_1+C_2M(t).
\ee
\end{proposition}
This proposition is attractive not only because it gives the desired bound on $h$, but it also gives an explicit lower bound on the Donaldson functional $M(t)$ that does not require existence of any canonical metric. However, in the case that $g$ is Gauduchon, we cannot generalize this result, since Simpson uses a form of the Donaldson functional given by integration by parts which we do not have access to.

Instead we adapt the $C^0$ bound from the elliptic approach of Uhlenbeck and Yau, suitably modified to fit our parabolic case. We note that both Simpson's result above, and the following proposition, rely on the theorem of Uhlenbeck and Yau that a weakly holomorphic $L^2$ subsheaf defines a coherent subsheaf. \begin{proposition}
\label{finalprop}
Let $H(t)$ be a solution of \eqref{DHF}, with $H_0$ suitably normalized so that ${\rm det}(h)=1$. Set 
\be
m(t):=\sup_X{\rm Tr}(h(t)).\nonumber
\ee
Suppose there does not exists a constant $C$ such that $m(t)<C$ uniformly in $t$. Then $(E,\t)$ is not stable.

\end{proposition}

We prove this proposition by contradiction, and assume no such constant exists. Define normalized endomorphisms $\ti h(t)={h(t)}/{m(t)}$ and consider $\ti h^\sigma(t)$ for any $0<\sigma\leq 1$.  We prove uniform $W^{1,2}$ bounds for $\ti h^\sigma(t)$, allowing us to construct  a weak limit after carefully selecting subsequences along the flow, which is important for several estimates. This weak limit is then used to construct a destabilizing subsheaf of $E$, contradicting stability of $E$.

First consider the following inequality, stated in Lemma (3.4.4) from \cite{TL}, which holds for $0<\sigma\leq 1$, and is in fact equality when $\sigma=1$.
\be
\label{int1}
|h^{-\frac\sigma2}\pl_{A_0} h^\sigma|^2_{H_0}\leq{i\mkern1mu}\Lambda{\rm Tr}(h^{-1}\pl_{A_0}h \bar\pl h^\sigma).
\ee
Note that we choose to work with the fixed covariant derivative $\pl_{A_0}$ and fixed metric $H_0$. The above inequality also uses the fact that $h$ is Hermitian with respect to $H_0$. Integrating the above inequality and integrating by parts yields
\bea
\label{aaa}
\int_X|h^{-\frac\sigma2}\pl_{A_0} h^\sigma|^2_{H_0}\,\frac{\o^n}{n!}&\leq&i\int_X{\rm Tr}(\bar\pl(h^{-1}\pl_{A_0} h)h^\sigma)\wedge\frac{\o^{n-1}}{(n-1)!}-i\int_X{\rm Tr}(h^{\sigma-1}\pl_{A_0}h)\wedge\frac{\bar\pl(\o^{n-1})}{(n-1)!}.\nonumber
\eea
Note the second term on the right vanishes since $X$ is Gauduchon
\bea
\int_X{\rm Tr}(h^{\sigma-1}\pl_{A_0}h)\wedge\bar\pl(\o^{n-1})&=&\frac1\sigma\int_X\pl{\rm Tr}(h^\sigma)\wedge{\bar\pl(\o^{n-1})}=-\frac1\sigma\int_X{\rm Tr}(h^\sigma)\,\pl\bar\pl(\o^{n-1})=0.\nonumber
\eea
Thus, applying formula \eqref{diffK}, we have
\be
\label{first}
\int_X|h^{-\frac\sigma2}\pl_{A_0} h^\sigma|^2_{H_0}\,\frac{\o^n}{n!}\leq\int_X{\rm Tr}((K-K_0)h^\sigma)\frac{\o^n}{n!}.
\ee
Now, we would like to bound the right hand side above by the $L^1$ norm of $h^\sigma$. However, Lemma \ref{supLF} gives a bound for the $C^0$ norm of $K_\t$ as opposed to $K$. To account for this, using the commutator and the properties of trace we have
\be
\label{tracetheta}
{\rm Tr}([\t,\t^\dagger]_\eta h^\sigma-[\t,\t^{\dagger_0}]_\eta h^\sigma)=\langle h^{-1}\t^{\dagger_0}(h),\t^\dagger(h^\sigma) \rangle_{H_0\otimes\bar\eta}.
\ee
We claim the following inequality
\be
\label{smallgoal}
|h^{-\frac\sigma2}\t^{\dagger_0}(h^\sigma)|^2_{H_0\otimes\bar\eta}\leq\langle h^{-1}\t^{\dagger_0}(h),\t^{\dagger_0}(h^\sigma) \rangle_{H_0\otimes\bar\eta},
\ee
again with equality in the case of $\sigma=1$. To see this, following the proof of Lemma (3.4.4) from \cite{TL}, fix a local frame for $E$ so that $h$ is diagonal with eigenvalues $e^{\lambda_i}$, with $1\leq i\leq rk(E)$. In this frame $\t^{\dagger_0}$ has a matrix representation where each entry of the matrix, denoted $\tau_{ij}$, is given by a section of $\bar V_0$. We then have
\be
|h^{-\frac\sigma2}\t^{\dagger_0}(h^\sigma)|^2_{H_0\otimes\bar\eta}=\sum_{i\neq j}e^{\sigma\lambda_i}(e^{\sigma\lambda_j}-e^{\sigma\lambda_i})^2|\tau_{ij}|^2_{\bar \eta},\nonumber
\ee
as well as the equality
\be
\langle h^{-1}\t^{\dagger_0}(h),\t^{\dagger_0}(h^\sigma) \rangle_{H_0\otimes\bar\eta}=\sum_{i\neq j}(e^{\lambda_j-\lambda_i}-1)(e^{\sigma\lambda_j}-e^{\sigma\lambda_i})|\tau_{ij}|^2_{\bar\eta}.\nonumber
\ee
Then inequality \eqref{smallgoal} follows from the fact that
\be
(e^\mu-1)(e^{\sigma\mu+\sigma\lambda}-e^{\sigma\lambda})\geq e^{-\sigma\lambda}(e^{\sigma\mu+\sigma\lambda}-e^{\sigma\lambda})^2\nonumber
\ee
for all real numbers $\lambda, \mu, \sigma$ with $0<\sigma\leq 1$. Combining \eqref{first} and \eqref{smallgoal} gives
\be
\int_X|h^{-\frac\sigma2}\pl_{A_0} h^\sigma|^2_{H_0}+|h^{-\frac\sigma2}\t^{\dagger_0}(h^\sigma)|^2_{H_0\otimes\bar\eta}\,\frac{\o^n}{n!}\leq\int_X{\rm Tr}((K_\t-K_0{}_\t)h^\sigma)\frac{\o^n}{n!}.\nonumber
\ee
Now, because $h$ is Hermitian with respect to $H$ (in addition to $H_0$), we have
\be
{\rm Tr}((K_\t-K_0{}_\t)h^\sigma)=\langle K_\t-K_0{}_\t,h^\sigma\rangle_H\leq|K_\t-K_0{}_\t|_H|h^\sigma|_H\nonumber
\ee
by the Cauchy-Schwarz inequality. We know $|K_\t|_H$ is bounded by Lemma \ref{supLF}, and because $h$ is Hermitian with respect to both $H$ and $H_0$ we know $|h^\sigma|_H=|h^\sigma|_{H_0}=\sqrt{{\rm Tr}(h^{2\sigma})}$. Thus
\be
\int_X|h^{-\frac\sigma2}\pl_{A_0} h^\sigma|^2_{H_0}+|h^{-\frac\sigma2}\t^{\dagger_0}(h^\sigma)|^2_{H_0\otimes\bar\eta}\,\frac{\o^n}{n!}\leq C\int_X{\rm Tr}|h^\sigma|_{H_0}\frac{\o^n}{n!}.\nonumber
\ee
Up to now we have only considered the unnormalized endomorphisms $h$, however one can divide both sides by $m^\sigma$ and the above inequality holds for $\ti h$. By definition of the normalization we have $\ti h\leq I$, which in turn implies $\ti h^{-\frac{\sigma}{2}}\geq I$. It follows that
\be
\label{bound}
\int_X|\pl_{A_0} \ti h^\sigma|^2_{H_0}+|\t^{\dagger_0}(\ti h^\sigma)|^2_{H_0\otimes\bar\eta}\,\frac{\o^n}{n!}\leq\int_X|\ti h^{-\frac\sigma2}\pl_{A_0} \ti h^\sigma|^2_{H_0}+|\ti h^{-\frac\sigma2}\t^{\dagger_0}(\ti h^\sigma)|^2_{H_0\otimes\bar\eta}\,\frac{\o^n}{n!}\leq C.
\ee
Thus for all $\sigma$, we have $\ti h^\sigma$ is bounded in $W^{1,2}$ uniformly in time, and this bound is independent of $\sigma$. Before we take a weak limit, we first carefully choose a subsequence of times along the flow. Define a sequence of powers $\sigma_j$ so that $\sigma_j\ra 0$ as $j\ra\infty$. For a fixed $j$ we define the function $f_j(t):\R\longrightarrow\R$ by $f_j(t):=\int_X{\rm Tr}(h^{\sigma_j})$. We argue now that because $m(t)$ is unbounded, the function $f_j(t)$ is unbounded as well.

We now turn to the following lemma
\begin{lemma}
\label{DNM}
Along the Donaldson heat flow, the following inequality holds uniformly in time
\be
m(t)\leq C\int_X{\rm Tr}(h)\frac{\o^n}{n!}.\nonumber
\ee
\end{lemma} 
\begin{proof}
The proof is identical to the K\"ahler case, and follows from the theory of elliptic PDE's. Taking the trace of \eqref{bootstrap}, and applying \eqref{int1} and \eqref{smallgoal} in the case of $\sigma=1$, gives
\bea
-P{\rm Tr}(h)&=&{\rm Tr}(h(K_\t-K_0{}_\t))-|h^{-\frac12}\t^{\dagger_0}(h)|^2_{H_0\otimes\bar\eta}-|h^{-\frac\sigma2}\pl_{A_0} h^\sigma|^2_{H_0}\nonumber\\
&\leq& |h|_H|K_\t-K_0{}_\t|_H\leq C{\rm Tr}(h),\nonumber
\eea
where the last inequality follows from Lemma \ref{supLF} the fact that the eigenvalues of $h$ are positive. We can now apply a standard Moser iteration argument for sub-solutions of elliptic equations, for example see Theorem 4.1 from \cite{FL} and set $p=1$. 
\end{proof}
Now, note that
\be
\int_X{\rm Tr}(h)\frac{\o^n}{n!}\leq\sup_X{\rm Tr}(h)^{1-\sigma_j}\int_X{\rm Tr}(h)^{\sigma_j}\frac{\o^n}{n!}\leq \frac12\sup_X{\rm Tr}(h)+2^{\frac{1-\sigma_j}{\sigma_j}}||{\rm Tr}(h)||_{L^{\sigma_j}}.\nonumber
\ee
Combining the above string of inequalities with Lemma \ref{DNM} we have
\be
\frac12 m(t)\leq (2C)^{\frac{1-\sigma_j}{\sigma_j}}||{\rm Tr}(h)||_{L^{\sigma_j}}.\nonumber
\ee
Because the function $(\cdot)^{\sigma_j}$ is an increasing function we take both sides to the power of $\sigma_j$ and see
\be
m(t)^{\sigma_j}\leq 2^{\sigma_j}(2C)^{1-\sigma_j}\int_X{\rm Tr}(h)^{\sigma_j}\frac{\o^n}{n!}.\nonumber
\ee
Let $r$ be the rank of $E$, and let $\lambda_1,...,\lambda_r$ be the positive eigenvalues of $h$. Again because the function $(\cdot)^{\sigma_j}$ is increasing we know that
\be
{\rm Tr}(h)^{\sigma_j}=(\lambda_1+...+\lambda_r)^{\sigma_j}\leq r^{\sigma_j}\lambda_{\rm Max}^{\sigma_j}\leq r^{\sigma_j}{\rm Tr}(h^{\sigma_j}).\nonumber
\ee
Thus putting everything together we have
\be
\label{nondegen}
m(t)^{\sigma_j}\leq (2r)^{\sigma_j}(2C)^{1-\sigma_j} f_j(t).
\ee
Since the left hand side becomes unbounded, we know for each $j$ the function $f_j(t)$ is unbounded. For fixed $j$ we pick a subsequence of times $t_{k(j)}$ with two properties. First, that $f_j(t_{k(j)})$ approaches infinity as $k(j)$ goes to infinity, and second that $\pl_t(f_j(t_{k(j)}))\geq 0$ for each $k(j)$. To see that this is possible let $t_{k(j)}$ be the time corresponding to $\sup_{t\in[0,k(j)]}f_j(t)$. Either this supremum occurs at a time $t<k(j)$, which means $f_j$ is at a local max in time, or it occurs at $k(j)$, which means $f_j$ must be non-decreasing in time, verifying the second desired property. Clearly this subsequence sends $f_j$ to infinity. Furthermore, because
\be
f_j(t)=\int_X{\rm Tr}(h^{\sigma_j})\frac{\o^n}{n!}\leq Vol(X)\sup_X{\rm Tr}(h^{\sigma_j})\leq rVol(X)m(t)^{\sigma_j},\nonumber
\ee
it is clear along the subsequence $t_{k(j)}$ that $m(t_{k(j)})$ goes to infinity as well. 

As stated before, along $t_{k(j)}$ we know $\ti h^\sigma$ is bounded in $W^{1,2}$. Thus, for each $\sigma_j$, there exists a subsequence of times (still denoted $t_{k(j)}$), that converges weakly in $W^{1,2}$ to a limiting endomorphism $h^{\sigma_j}_\infty$. In fact this limiting endomorphism is non-degenerate, which we can see by dividing \eqref{nondegen} by $m^{\sigma_j}$
\be
1\leq(2r)^{\sigma_j}(2C)^{1-\sigma_j}\int_X{\rm Tr}(\ti h^{\sigma_j})\frac{\o^n}{n!}\leq Vol(X)(2r)^{\sigma_j}(2C)^{1-\sigma_j}||\ti h^{\sigma_j}||_{L^2}^2(X).\nonumber
\ee
Thus the $L^2$ norm of $\ti h^{\sigma_j}$ is bounded below independent of the subsequence $t_{k(j)}$, and since weak $W^{1,2}$ convergence implies strong convergence in $L^2$ we know the $L^2$ norm of $h^{\sigma_j}_\infty$ is non zero.

Note that the endomorphisms $\ti h^{\sigma_j}_\infty$ are bounded uniformly in $W^{1,2}$ for all $j$ (which follows because \eqref{bound} is independent of $\sigma$). We therefore have an $W^{1,2}$ limit along a subsequence (still denoted $j$) as $\sigma_j$ approaches zero, which converges to a limit $h^0_\infty$. This limit is also non-degenerate, since the constant $(2r)^{\sigma_j}(2C)^{1-\sigma_j}$ from \eqref{nondegen} approaches $2C$ as $\sigma\ra 0$, which is bounded.

In order to construct a destabilizing sub-sheaf, we want to invoke a Theorem of Uhlenbeck and Yau from $\cite{UY}$ which states that a weakly holomorphic $W^{1,2}$ projection actually defines a coherent sub sheaf of $E$. Our weakly holomorphic projection is defined as follows
\be
\label{pi}
\pi=\lim_{j\ra \infty}\lim_{k\ra\infty}(I-\ti h^{\sigma_j}(t_{k(j)})).
\ee
The endomorphism $\pi$ is in $W^{1,2}$ since the endomorphisms $\ti h^{\sigma_j}_{k(j)}$ converge weakly in $W^{1,2}$. Following the argument in \cite{TL} one easily checks that $\pi^*=\pi^2=\pi$. Finally, to apply the theorem from  $\cite{UY}$  one needs that $(I-\pi)(\bar\pl+\t)\pi=0$ in $L^1$. This will prove not only that $\pi$ is a weakly holomorphic subbundle, but also that $\pi$ is preserved by the Higgs field $\theta$. To see this fact, by the argument on the bottom of page 86 of \cite{TL}, it suffices to show
\be
\label{fact}
||\pi\pl_{A_0}(I-\pi)||^2_{L^2(H_0)}=||\pi\t^{\dagger_0}(I-\pi)||^2_{L^2(H_0\otimes\bar\eta)}=0. 
\ee
Now, for all real numbers $0\leq\lambda\leq1$ and $0<s\leq\sigma\leq 1$, we have the following bound from $\cite{UY}$
\be
0\leq\frac{s+\sigma}s(1-\lambda^s)\leq\lambda^{-\sigma}.\nonumber
\ee
Thus for $0<s\leq\frac{\sigma_j}2\leq 1$ it holds
\be
0\leq\frac{s+\frac{\sigma_j}2}s(I-\ti h^s(t_{k(j)}))\leq \ti h^{-\frac{\sigma_j}2}(t_{k(j)}).\nonumber
\ee
By \eqref{bound} one now has
\bea
&&\int_X\left|\left(I-\ti h^s\right)\pl_{A_0} \ti h^{\sigma_j}\right|^2_{H_0}+\left|\left(I-\ti h^s\right)\t^{\dagger_0} \ti h^{\sigma_j}\right|^2_{H_0\otimes\bar\eta}\,\frac{\o^n}{n!}\nonumber\\
&&\leq\frac s{s+\frac{\sigma_j}2}\int_X|\ti h^{-\frac{\sigma_j}2}\pl_{A_0} \ti h^{\sigma_j}|^2_{H_0}+|\ti h^{-\frac{\sigma_j}2}\t^{\dagger_0}(\ti h^{\sigma_j})|^2_{H_0\otimes\bar\eta}\,\frac{\o^n}{n!}\leq C\frac s{s+\frac{\sigma_j}2},\nonumber
\eea
where the dependence of $\ti h$ on the time $t_{k(j)}$ has been suppressed for notational simplicity. Then \eqref{fact} follows by first letting $k$, then $s$, then $j$ go to infinity. We direct the reader \cite{TL, UY} for details, which are the same as in the K\"ahler case. We can now apply the following theorem of Uhlenbeck and Yau from \cite{UY}.
 \begin{theorem}
 \label{UY-}
 Given a weakly holomorphic subbundle $\pi$ of $E$, there exists a coherent subsheaf $\mathcal{F}$ of $E$,  and an analytic subset $S\subset X$ with the following properties
 
 \medskip
 i) $codim_XS\geq 2$
 
 ii) $\pi|_{X\backslash S}$ is $\Ci$ and satisfies both $\pi^*=\pi=\pi^2$ and $(I-\pi)\bar\pl\pi=0$
 
 iii) $\F':=\F|_{X\backslash S}$ is a holomorphic subbundle, and on $X\backslash S$ the endomorphism $\pi$ is the projection of $E$ onto $\F'$.
 \end{theorem}
Thus we have  constructed a coherent subsheaf $\mathcal{F}$ of $E$, and to finish Proposition $\ref{finalprop}$ we must show $\cal F$ is proper and destabilizing. We first show  $\F$ is a proper subsheaf of $E$.

 Since $h^0_\infty$ is non-degenerate, it must have at least one nonzero eigenvalue. Thus $rk\,(h^0_\infty)\geq1$ which implies
 \be
 rk(\F)=rk(\pi)=rk\, (I-h^0_\infty)\leq r-1.\nonumber
 \ee
On the other hand we are assuming that $m(t_{k(j)})$ goes to infinity along every subsequence $k(j)$ (we have explicitly noted this is true for all $j$).  Because det$(h(t))=1$ along the Donaldson heat flow we must have an eigenvalue of $\ti h^{\sigma_j}_{k(j)}$ that goes to zero. Thus almost everywhere $h^{\sigma_j}_\infty$ has an eigenvalue equal to zero, and by strong $L^2$ convergence almost everywhere $h^0_\infty$ has an eigenvalue equal to zero, which implies $rk (\F)>0$. So $\F$ is indeed a proper subsheaf of $E$. 
 
We now prove $\mu(\F)\geq\mu(E)$ showing that $\F$ is destabilizing. Recall the Chern-Weil formula from Section $\ref{HB}$, which we apply using the fixed connection associated to $H_0$
 \be
 \mu(\F)=\frac{1}{2\pi \,rk(\F)}\left(\int_X{\rm Tr}(\pi K_0 \pi)\,\frac{\o^n}{n!}-||\pl_{A_0}\pi||^2_{L^2(H_0)}\right).\nonumber
  \ee
Using the definition of $\lambda$, given by \eqref{lambda}, we modify the formula slightly to include $\mu(E)$
\be
 \mu(\F)=\frac{1}{2\pi rk(\F)}\left(\int_{X}{\rm Tr}((K_0-\lambda I)\circ\pi)\,\frac{\o^n}{n!}-||\pl_{A_0}\pi||^2_{L^2(H_0)}\right)+\mu(E).\nonumber
 \ee
Thus to show $\mu(\F)\geq\mu(E)$, we must verify
  \be
  \label{semifin}
  \int_{X}{\rm Tr}((K_0-\lambda I)\circ\pi)\,\frac{\o^n}{n!}\geq||\pl_{A_0}\pi||^2_{L^2(H_0)}.
  \ee
The inequality above is a direct consequence of the following lemma.
\begin{lemma}
  \label{finaleq}
  The projection $\pi$ defined by \eqref{pi} satisfies the following inequality
 \be
  \int_{X}{\rm Tr}((K_0{}_\t-\lambda I)\circ\pi)\,\frac{\o^n}{n!}\geq||\pl_{A_0}\pi||^2_{L^2(H_0)}+||\t^{\dagger_0}(\pi)||^2_{L^2(H_0\otimes\bar\eta)}.\nonumber
 \ee
\end{lemma}
Note that the left had side of the above inequality contains $K_0{}_\t$ as opposed to $K_0$. However, one can show
  \be
  {\rm Tr}([\t,\t^{\dagger_0}]_\eta\pi)=|\t^{\dagger_0}(\pi)|^2_{H_0\otimes\bar\eta}.\nonumber
  \ee
  To see this, we use that $\t$ preserves the subbundle defined by $\pi$, which means $(I-\pi)\t\pi=0$ in $L^1$. Using the commutator and the fact that $\pi\t\pi=\t\pi$ in the $L^1$ sense it is easy to verify the above equality. Thus Lemma \ref{finaleq} is indeed equivalent to \eqref{semifin}.

We now prove Lemma \ref{finaleq}. By the definition of degree we have $\int_X$Tr$(K_0+[\t,\t^{\dagger_0}]_\eta-\lambda I)=0$. This fact, along with the observation that the convergence defining $\pi$ is strong in $L^2$, yields
 \be
 \int_{X}{\rm Tr}((K_0{}_\t-\lambda I)\circ\pi)\,\frac{\o^n}{n!}=-\lim_{j\ra\infty}\lim_{k\ra\infty}\int_X{\rm Tr}((K_0{}_\t-\lambda I)\,\ti h^{\sigma_j}(t_{k(j)}))\,\frac{\o^n}{n!}.\nonumber
 \ee
From now on we drop the $t_{k(j)}$ from our expressions to ease notation. Our next step is to modify the above formula so that it contains $K_\t$ instead of $K_0{}_\t$. Applying equation \eqref{diffK} gives
 \bea
 \label{cool}
 -\int_X{\rm Tr}((K_0{}_\t-\lambda I)\,\ti h^{\sigma_j})\frac{\o^n}{n!}&=&\int_X{\rm Tr}(\bar\pl(\ti h^{-1} \pl_{A_0}\ti h)\,\ti h^{\sigma_j})\wedge\frac{\o^{n-1}}{(n-1)!}\\
&&+\int_X{\rm Tr}\left(-(K_\t-\lambda I)\,\ti h^{\sigma_j}+[\t,\t^\dagger]_\eta \ti h^\sigma-[\t,\t^{\dagger_0}]_\eta \ti h^\sigma\right)\,\frac{\o^n}{n!}.\nonumber
  \eea
Now, by our flow equation \eqref{DHF} we have
\be
-\int_X{\rm Tr}((K_\t-\lambda I)\,\ti h^{\sigma_j})\,\frac{\o^n}{n!}=\frac{1}{m(t_{k(j)})^{\sigma_j}}\int_X{\rm Tr}(h^{-1}\dot h h^{\sigma_j})\,\o^n=\frac{\pl_t f_j(t_{k(j)})}{\sigma_j m(t_{k(j)})^{\sigma_j}}.\nonumber
\ee
Both $m(t)$ and $\sigma_j$ are always positive, and for each fixed $j$ we chose our subsequence $k(j)$ so that $\pl_t f_j(t_{k(j)})\geq 0$. Furthermore, we have already seen that
\bea
\frac1{m^{\sigma_j}}{\rm Tr}([\t,\t^\dagger]_\eta h^\sigma-[\t,\t^{\dagger_0}]_\eta h^\sigma)&\geq& \frac1{m^{\sigma_j}}|h^{-\frac\sigma2}\t^{\dagger_0}(h^\sigma)|^2_{H_0\otimes\bar\eta}\nonumber\\
&=&|\ti h^{-\frac\sigma2}\t^{\dagger_0}(\ti h^\sigma)|^2_{H_0\otimes\bar\eta}\geq |\t^{\dagger_0}(\ti h^\sigma)|^2_{H_0\otimes\bar\eta}.\nonumber
\eea
Thus we can return to \eqref{cool} and conclude
  \be
 -\int_X{\rm Tr}((K_0{}_\t-\lambda I)\,\ti h^{\sigma_j})\frac{\o^n}{n!}\geq \int_X{\rm Tr}(\bar\pl(\ti h^{-1} \pl_{A_0}\ti h)\,\ti h^{\sigma_j})\wedge\frac{\o^{n-1}}{(n-1)!}+||\t^{\dagger_0}(\ti h^\sigma)||^2_{L^2(H_0\otimes\bar\eta)}.\nonumber
   \ee
Note that the first term on the right hand side above appears in proof of Proposition 3.4.8 from \cite{TL}. Following their argument exactly, one can integrate by parts to prove the following inequality
\be
 \int_X{\rm Tr}(\bar\pl(\ti h^{-1} \pl_{A_0}\ti h)\,\ti h^{\sigma_j})\wedge\frac{\o^{n-1}}{(n-1)!}\geq\int_X|\ti h^{-\frac{\sigma_j}2}\pl_{A_0} \ti h^{\sigma_j}|_{H_0}^2\,\frac{\o^n}{n!}\geq||\pl_{A_0}(\ti h^{\sigma_j})||^2_{L^2(H_0)}.
\nonumber
\ee
Putting everything together, so far we have
\be
-\lim_{j\ra\infty}\lim_{k\ra\infty}\int_X{\rm Tr}((K_0{}_\t-\lambda I)\,\ti h^{\sigma_j})\frac{\o^n}{n!}\geq \lim_{j\ra\infty}\lim_{k\ra\infty}\left(||\pl_{A_0}(\ti h^{\sigma_j})||^2_{L^2(H_0)}+||\t^{\dagger_0}(\ti h^\sigma)||^2_{L^2(H_0\otimes\bar\eta)}\right).\nonumber
\ee
Because convergence of $\ti h^{\sigma_j}_{k(j)}$ is strong in $L^2$, by first taking the limit in $k$ and then $j$ the left hand side can be written as
\be
\lim_{j\ra\infty}\lim_{k\ra\infty}\int_X{\rm Tr}\left((K_0{}_\t-\lambda I)\,(I-\ti h^{\sigma_j})\right)\frac{\o^n}{n!}=\int_{X}{\rm Tr}((K_0{}_\t-\lambda I)\circ\pi)\,{\o^n}.\nonumber
\ee
Now, convergence is only weak in $W^{1,2}$, yet by lower semi-continuity of weak limits we have
\be
||\pl_{A_0}(\pi)||^2_{L^2(H_0)}+||\t^{\dagger_0}(\pi)||^2_{L^2(H_0\otimes\bar\eta)}\leq\lim_{j\ra\infty}\lim_{k\ra\infty}\left(||\pl_{A_0}(\ti h^{\sigma_j})||^2_{L^2(H_0)}+||\t^{\dagger_0}(\ti h^\sigma)||^2_{L^2(H_0\otimes\bar\eta)}\right).\nonumber
\ee
From here Lemma \ref{finaleq} follows, which verifies that $\F$ is destabilizing. This completes the proof of Proposition $\ref{finalprop}$.

 \end{normalsize}

\vspace{10mm}

\begin{centering}

\textnormal{ Department of Mathematics,
Harvard University\\
One Oxford Street\\
Cambridge, MA 02138\\
e-mail: ajacob@math.harvard.edu}

\end{centering}


\newpage

\begin{thebibliography}{4}

{\small

\bibitem{Bis} I. Biswas, {\it Stable Higgs bundles on compact Gauduchon manifolds},
C.R. Math. Acad. Sci. Paris {\bf 349} (2011), no. 1-2, 71-74.

\bibitem{Ban} D. Banfield, {\it The geometry of coupled equations in gauge theory}, D. Phil. Thesis,
University of Oxford (1996).

\bibitem{BGM} S. Bradlow, O. Garcia-Prada and I. Mundet i Riera,

\bibitem{Buch2} N.P. Buchdahl,
{\it Hermitian-Einstein connections and stable vector bundles over compact complex surfaces,} Math. Ann., {\bf 280} (1988), no. 4, 625-684.

\bibitem{CJ} T.C. Collins and A. Jacob,
{\it Remarks on the Yang-Mills flow on a compact K\"ahler manifold,}
Univ. Iagel. Acta Math. (to appear).

\bibitem{Cor1} K. Corlette, {\it Flat G-bundles with canonical metrics,} J. Differential Geom. {\bf 28} (1988), no. 3, 361-382

\bibitem{Cor2} K. Corlette, {\it Nonabelian Hodge theory,} Proce. Sympos. Pure Math., {\bf 54} part 2,  Amer. Math. Soc., Providence, RI  (1993)

\bibitem{DO} K. Diederich and T. Ohsawa,
{\it Harmonic mappings and disc bundles over compact K\"ahler manifolds,}
Publ. Res. Inst. Math. Sci. {\bf 21} (1985), 819-833.

\bibitem{Don1} S.K. Donaldson,
{\it Anti self-dual Yang-Mills connections over complex angebraic surfaces and stable vector bundles,}
Proc. London Math. Soc. (3) {\bf 50} (1985), 1-26.

\bibitem{Don3} S.K. Donaldson,
{\it Twisted harmonic maps and the self-duality equations,}
Proc. London Math. Soc. {\bf 55} (1987), 127-131.

\bibitem{G} P. Gauduchon,
{\it Sur la 1-forme de torsion dÕune vari\'et\'e hermitienne compacte}, Math. Ann. 267, 495-518 (1984)

\bibitem{GH} P. Griffiths and J. Harris, 
{\it Principles of Algebraic Geometry},
John Wiley $\&$ Sons, (1978).

\bibitem{Ham} R.S. Hamilton, {\it Three-manifolds with positive Ricci curvature,} J. Differential Geom. {\bf 17} (1982), no. 2, 255-306.

\bibitem{FL} Q. Han and F. Lin, {\it Elliptic partial differential equations,} Second edition. Courant Lecture Notes in Mathematics. Courant Institute of Mathematical Sciences, New York; American Mathematical Society, Providence, RI, (2011)


\bibitem{J} A. Jacob, {\it Existence of approximate Hermitian-Einstein structures on semi-stable bundles,}  Asian J. Math. (to appear).

\bibitem{Kob} S. Kobayashi, 
{\it Differential geometry of complex vector bundles,}
Publications of the Mathematical Society of Japan, {\bf 15}. Kan\^o Memorial Lectures, {\bf 5}. Princeton University Press, Princeton, NJ (1987).

\bibitem{LY} J. Li and S.-T. Yau,
{\it Hermitian-Yang-Mills connections on non-K\"ahler manifolds,}
Mathematical aspects of string theory (San Diego, Calif., 1986), 560-573, World Sci. Publishing, Singapore, (1987).

\bibitem{Lib} G. Liberman, {\it Second order parabolic differential equations}, World Scientific Publishing Co., Inc., River Edge, NJ, (1996).

\bibitem{Lub} M. L\"ubke, {\it Einstein metrics and stability for flat connections on compact Hermitian manifolds, and a correspondence with Higgs operators in the surface case.} Doc. Math. {\bf 4} (1999), 487-512.

\bibitem{TL} M. L\"ubke and A. Teleman,
{\it The Kobayashi-Hitchin correspondence,}
World Sci. Publ., River Edge, NJ (1995).

\bibitem{TL2} M. L\"ubke and A. Teleman, {\it The universal Kobayashi-Hitchin correspondence on Her-
mitian manifolds}, Mem. Amer. Math. Soc., Vol. 183, No. 863, July (2006).

\bibitem{Donovan} D. McFeron,
{\it Remarks on some non-linear heat flows in K\"ahler geometry,} Ph.D. Thesis, Columbia University (2009).

\bibitem{M} I. Mundet i Riera, {\it  A Hitchin-Kobayashi correspondence for K¬ahler fibrations}, J. Reine Angew. Math. {\bf 528}, 41-80 (2000).

\bibitem{NS} M.S. Narasimhan and C.S. Seshadri,
{\it Stable and unitary bundles on a compact Riemann surface,} Math. Ann. {\bf 82} (1965), 540-564.


\bibitem{P1} G. Perelman, {\it Finite extinction time for the solutions of the Ricci flow on certain three-manifolds,} preprint 2003, arXiv:math/0307245.

\bibitem{P2} G. Perelman, {\it Ricci flow with surgery on three-manifolds,} preprint 2003, arXiv:math/0303109.

\bibitem{P3} G. Perelman, {\it The entropy formula for the Ricci flow and its geometric applications,} preprint 2002, arXiv:math/0211159.

\bibitem{PSSW} D.H. Phong, J. Song, J. Sturm, and B. Weinkove,
{\it The K\"ahler-Ricci flow and the $\bar\pl$ operator on vector fields},
J. Differential Geom.  {\bf 81}  (2007),  no. 3, 631-647.

\bibitem{Samp} J. Sampson,
{\it Applications of harmonic maps to K\"ahler geometry,}
Contemp. Math. no. 49, Amer. Math. Soc., Rovidence, RI (1986), 125-133 

\bibitem{Simp} C. Simpson,
{\it Constructing variations of Hodge structure using Yang-Mills theory and applications to uniformization,}
J. Amer. Math. Soc.  {\bf 1}  (1988),  no. 4, 867-918.

\bibitem{Simp2} C. Simpson, {\it Higgs bundles and local systems,} Inst. Hautes ƒtudes Sci. Publ. Math. {\bf 75} (1992), 5-95.

\bibitem{Siu}Y.-T. Siu,
{\it Lectures on Hermitian-Einstein metrics for stable bundles and K\"ahler-Einstein metrics,}
Birk\"auser Verlag, Basel, (1987).

\bibitem{Siu2}Y.-T. Siu,
{\it The complex analyticity of harmonic maps and the strong rigity of compact K\"ahler manifolds,}
Ann. of Math, (2) {\bf 112} (1980), 73-112.

\bibitem{UY}K. Uhlenbeck and S.-T. Yau,
{\it On the existence of Hermitian-Yang-Mills connections in stable vector bundles,}
Comm. Pure and Appl. Math. {\bf 39-S} (1986), 257-293.

\bibitem{W2}B. Weinkove,
{\it Singularity formation in the Yang-Mills flow},
Calc. Var. Partial Differential Equations. {\bf19}, no. 2, (2004), 211-220.

}
\end{thebibliography}
\end{document}